\documentclass[letterpaper, 10 pt, conference]{ieeeconf}  

\IEEEoverridecommandlockouts                              

\usepackage{todonotes}
\usepackage{latexsym,amssymb,amsmath, amsbsy,amsopn, amstext}
\usepackage{graphicx,epsfig,tikz,pgf,hyperref,color,ifpdf}
\usepackage{mathrsfs}
\usepackage{amsmath,  amssymb,stfloats}
\usepackage{subcaption,enumerate}
\usepackage{tikz}
\usepackage{cite}
\usepackage[ruled,vlined,linesnumbered]{algorithm2e}
\usepackage{bm}
\usepackage{etoolbox}
\makeatletter
\patchcmd{\@makecaption}
  {\scshape}
  {}
  {}
  {}
\makeatother
\usepackage{makecell}

\makeatletter
\patchcmd{\@algocf@start}
  {-1.5em}
  {0pt}
  {}{}
\makeatother

\graphicspath{{./}}

\newtheorem{definition}{\bfseries Definition}
\newtheorem{proposition}{\bfseries Proposition}
\newtheorem{example}{\bfseries Example}
\newtheorem{theorem}{\bfseries Theorem}
\newtheorem{lemma}{\bfseries Lemma}

\newtheorem{remark}{\bfseries Remark}

\DeclareMathOperator{\itl}{interval}

\def\v{\bm{v}}

\def\x{\bm{x}}

\def\u{\bm{u}}

\def\h{\bm{h}}
\def\y{\bm{y}}
\def\W{\bm{W}}
\newcommand{\mm}[1]{\mathcal{#1}}
\newcommand{\mt}[1]{\boldsymbol{#1}}

\newcommand{\li}{[\![}
\newcommand{\ri}{]\!]}

\newcommand{\cc}{\mathbf{c}}
\newcommand{\G}{\mathbf{G}}
\newcommand{\A}{\mathbf{A}}
\newcommand{\bb}{\mathbf{b}}

\newcommand{\Zc}{\mathcal{Z}}
\newcommand{\X}{\mathcal{X}}

\newcommand{\mat}[1]{\begin{bmatrix} #1 \end{bmatrix}}
\newcommand{\hz}[1]{\langle \G^c_{#1},\allowbreak \G^b_{#1},\allowbreak \cc_{#1},\allowbreak \A^c_{#1},\allowbreak \A^b_{#1},\allowbreak \bb_{#1} \rangle}

\newcommand{\lrangle}[1]{\langle #1 \rangle}

\newcommand{\seqt}[2]{\{{#1}_{t}\}_{t=1}^{#2}}

\allowdisplaybreaks[4]

\usepackage{xcolor}
\newif\ifdraft

\usepackage{caption}

\drafttrue

\title{\LARGE \bf Forward and Backward Reachability Analysis of Closed-loop Recurrent Neural Networks via Hybrid Zonotopes}

\author{Yuhao Zhang and Xiangru Xu
\thanks{Y. Zhang and X. Xu are with the Department of Mechanical Engineering, University of Wisconsin-Madison, Madison, WI, USA. Email:         {\tt\small \{yuhao.zhang2,xiangru.xu\}@wisc.edu}.}%
}

\begin{document}
\maketitle
\begin{abstract}
Recurrent neural networks (RNNs) are widely employed to model complex dynamical systems due to their hidden-state structure, which inherently captures temporal dependencies. This work presents a hybrid zonotope–based approach for computing exact forward and backward reachable sets of closed-loop RNN systems with ReLU activation functions. The method formulates state-pair sets to compute reachable sets as hybrid zonotopes without requiring unrolling. To improve scalability, a tunable relaxation scheme is proposed that ranks unstable ReLU units across all layers using a triangle-area score and selectively applies convex relaxations within a fixed binary limit in the hybrid zonotopes. This scheme enables an explicit trade-off between computational complexity and approximation accuracy, with exact reachability as a special case. In addition, a sufficient condition is derived to certify the safety of closed-loop RNN systems. Numerical examples demonstrate the effectiveness of the proposed approach.

\end{abstract}

\section{Introduction}\label{sec:intro}








Neural networks (NNs) have found widespread use in various engineering applications due to their universal function approximation capabilities, which allow them to represent complex functions and system dynamics effectively \cite{hornik1989multilayer}. Unlike feedforward and convolutional neural networks that process inputs independently, Recurrent neural networks (RNNs)  maintain hidden states that explicitly encode temporal dependencies, making them well-suited for sequential tasks such as natural language processing \cite{liu2016recurrent}, speech recognition \cite{sak2014long}, and trajectory prediction \cite{alahi2016social}. In control-related applications, RNNs have been widely employed for modeling uncertain dynamics, estimating hidden states, and controlling dynamical systems, where outputs depend not only on current inputs but also on historical context \cite{bonassi2022recurrent}.

Despite their increasing popularity, RNNs are known to suffer from issues such as the exploding gradient problem \cite{bengio1994learning} and high sensitivity to input perturbations \cite{cheng2020seq2sick}. However, the theoretical properties of RNNs, such as robustness, verifiability, and interpretability, remain relatively underexplored, raising significant safety concerns when RNNs are deployed in safety-critical applications. In the dynamical system settings, recent research has primarily focused on stability analysis of RNNs. For example, \cite{d2023incremental} established incremental input-to-state stability conditions of RNNs and employed linear matrix inequalities to design stable RNN-based controllers. In \cite{revay2021convex}, a convex parameterization was developed to guarantee stability and robustness via incremental quadratic constraints, while regional stability conditions were studied in \cite{la2025regional}. 

Beyond stability, the verification and robustness analysis of isolated RNNs has also been investigated through reachability analysis. Two main approaches have been considered: unrolling-based methods, which expand an RNN into a large feedforward neural network (FNN), and invariant inference methods, which construct smaller FNN over-approximations. The unrolling approach suffers from poor scalability due to the rapid growth in network size with time steps \cite{akintunde2019verification}, whereas invariant inference offers better scalability but may produce inconclusive results because of accumulated over-approximation errors \cite{jacoby2020verifying}. To address these limitations, a set-propagation-based method using star sets was proposed in \cite{tran2023verification}, with scalability further improved through sparse star sets in \cite{choi2025reachability}.   Despite these advances, backward reachability analysis of RNNs, which is crucial for identifying adversarial input sequences \cite{papernot2016crafting} and enabling safe control synthesis for neural network control systems \cite{zhang2025goalreaching}, remains largely unexplored in the literature.

This work introduces a novel set-based method for computing both forward and backward reachable sets of closed-loop RNN systems using hybrid zonotopes (HZs). The main contributions are threefold: (i) We compute exact forward and backward reachable sets for closed-loop RNNs with ReLU activation functions, represented as HZs, by formulating state-pair sets  without unrolling the RNNs. (ii) To improve scalability, we propose a tunable relaxation scheme that manages the complexity of the resulting HZ-represented reachable sets, by ranking unstable ReLU units across all layers using a triangle-area score and selectively applying convex relaxations  within a fixed limit on binary variables in the hybrid zonotopes. This enables  an explicit trade-off between computational complexity and approximation accuracy, with exact reachability as a special case. (iii) Using the computed reachable sets, we derive a sufficient condition for safety verification and for identifying unsafe sequences in closed-loop RNNs. Numerical examples demonstrate the effectiveness and tunability of the proposed method.

\emph{Notation.} 
The $i$-th component of a vector $\x\in \mathbb{R}^n$ is denoted by $x_{i}$ with $i\in [n]\triangleq \{1,\dots,n\}$. 
The identity matrix in $\mathbb{R}^{n\times n}$ is denoted as $\bm{I}_n$ and $\bm{e}_i$ is the $i$-th column of $\bm{I}_n$. 
The matrices with all zero entries are denoted as $\bm{0}$. 
Given sets $\mathcal{X}\subset \mathbb{R}^n$, $\mathcal{Z}\subset \mathbb{R}^m$ and a matrix $\bm{R}\in\mathbb{R}^{m\times n}$, 
the generalized intersection of $\mathcal{X}$ and $\mathcal{Z}$ under $\bm R$ is $\mathcal{X} \cap_{\bm R}\mathcal{Z} = \{\x\in\mathcal{X}\;|\;\bm R \x\in\mathcal{Z}\}$. 
The interval hull of a set $\mathcal{X}\subset \mathbb{R}^n$ is denoted as $\itl{(\mathcal{X})}\subset \mathbb{R}^n$. 
The affine transformation of $\mm{X}$ given a matrix $\mt{A}$ and vector $\mt{b}$ is $\mt{A}\mm{X}+\mt{b} = \{\mt{Ax} + \mt{b} \mid \mt{x} \in \mm{X}\}$.
An interval with bounds $\underline{\mt{a}}$, $\overline{\mt{a}} \in \mathbb{R}^n$ is denoted as $\li \underline{\mt{a}}, \overline{\mt{a}}\ri$. 

\section{Preliminaries \& Problem Statement}\label{sec:pre}

\subsection{Hybrid Zonotopes}\label{sec:zono}
First, we provide the formal definition of hybrid zonotope (HZ). 
A set $\Zc \subset\mathbb{R}^n$ is a HZ if there exist $\mathbf{c} \in \mathbb{R}^{n}$, $\mathbf{G}^c \in \mathbb{R}^{n \times n_{g}}$, $\G^b\in \mathbb{R}^{n \times n_{b}}$, $\A^c \in \mathbb{R}^{n_{c}\times{n_g}}$, $\A^b \in\mathbb{R}^{n_{c}\times{n_b}}$, $\bb \in \mathbb{R}^{n_{c}}$ such that $\mathcal{Z}=\{ \G^c  \bm{\xi}^c+\G^b \bm{\xi}^b+\cc \mid 
\bm{\xi}^c\in \mathcal{B}_{\infty}^{n_{g}}, \bm{\xi}^b\in\{-1,1\}^{n_{b}}, 
\A^c\bm{\xi}^c+\A^b \bm{\xi}^b=\bb
\}$ where $\mathcal{B}_{\infty}^{n_g}=\left\{\bm{x} \in \mathbb{R}^{n_g} \;|\;\|\bm x\|_{\infty} \leq 1\right\}$ is the unit hypercube in $\mathbb{R}^{n_{g}}$ \cite[Definition 3]{bird2023hybrid}. The constrained generator representation of $\Zc$  is given by $\mathcal{Z}= \hz{}$, where the columns of $\G^b$ are called the binary generators and the columns of $\G^c$ are called the continuous generators. 

In this work, we adopt the identity given in the following lemma for all generalized intersections of HZs.

\begin{lemma}\label{lemma:inter}
For any HZs $\mathcal{Z}=\hz{z} \subset \mathbb{R}^{n}$ and $\mathcal{Y}=\hz{y} \subset \mathbb{R}^{m}$, and matrix $\bm R \in \mathbb{R}^{m \times n}$, the generalized intersection of $\mathcal{Z}$ and  $\mathcal{Y}$ under $\bm R$ is given by:
\end{lemma}
\begin{align}
&\mathcal{Z} \cap_{\bm R} \mathcal{Y} 
=  \langle\mat{ \mathbf{0}& \G_{z}^{c} },\mat{ \mathbf{0}& \G_{z}^{b} }, \cc_{z}, \nonumber\\
&\quad \mat{\A_{y}^{c} & \mathbf{0} \\ \mathbf{0} & \A_{z}^{c} \\ -\G_{y}^{c} & \bm{R}\G_{z}^{c} } ,\mat{\A_{y}^{b} & \mathbf{0} \\\mathbf{0} & \A_{z}^{b} \\ -\G_{y}^{b} & \bm{R}\G_{z}^{b} },
\mat{\bb_{y} \\
\bb_{z} \\
\cc_{y}-\bm{R}\cc_{z} }\rangle.\label{eqhzinter}
\end{align}

The identity for the generalized intersection of HZs given in \eqref{eqhzinter} is slightly different from that given in \cite[Proposition 7]{bird2023hybrid}. Specifically, the equality constraints of the second operand (i.e., $\A^c_y$, $\A^b_y$, and $\bb_y$) are positioned at the top of the constructed matrices. However, it is easy to verify that the HZs generated by these two identities are equivalent.

\subsection{Recurrent Neural Networks}\label{subsec:RNN}
Given an input sequence $\seqt{\x}{T}$ with $\bm{x}_t \in \mathbb{R}^n$, an RNN $\mt{\pi}$ 
with $L$ hidden layers (as shown in Figure \ref{fig:rnn} (a)) computes a sequence of outputs $\seqt{\y}{T}$, for $\ell \in [L]$ and $t \in [T]$, according to
\begin{subequations}\label{equ:RNN}
\begin{align}
    \bm{h}^{(0)}_t & = \x_t, \label{equ:RNN-1}\\
    \bm{h}_t^{(\ell)} &= \sigma_h^{(\ell)}(\bm W_h^{(\ell)} \bm{h}_{t-1}^{(\ell)} + \bm W_x^{(\ell)} \bm{h}_{t}^{(\ell-1)} + \bm{v}_h^{(\ell)} ), \label{equ:RNN-2}\\
    \bm{y}_t &= \sigma_y(\bm W_y \bm{h}_t^{(L)} + \bm{v}_y ),
\end{align}
\end{subequations}
where $\bm{h}_t^{(\ell)}\in \mathbb{R}^{n_\ell}$ is the $t$-step hidden state at the $\ell$-th layer; $\bm W_h^{(\ell)}\in \mathbb{R}^{n_\ell\times n_\ell}$, $\bm W_x^{(\ell)}\in \mathbb{R}^{n_\ell\times n_{\ell-1}}$ and $\bm v_h^{(\ell)}\in \mathbb{R}^{n_\ell}$ are the $\ell$-th layer hidden state matrix, input matrix and hidden bias vector, respectively; $\bm W_y \in \mathbb{R}^{m\times n_L}$ and $\bm{v}_y\in \mathbb{R}^{m}$ are the output matrix and output bias vector; $\sigma_h^{(\ell)}$ and $\sigma_y$ are the activation functions. In this work, we assume all the initial hidden states are zeros (i.e., $\bm{h}_0^{(\ell)} = \bm 0$) and all the activation functions are ReLU functions. However, the proposed methods can be easily extended to non-zero initial hidden states and other types
of activation functions by using their HZ approximation as in \cite{zhang2024hybrid}.




\begin{figure}[!t]
    \centering
    \begin{subfigure}[t]{0.2145\linewidth}
        \centering
        \includegraphics[width=0.9\linewidth]{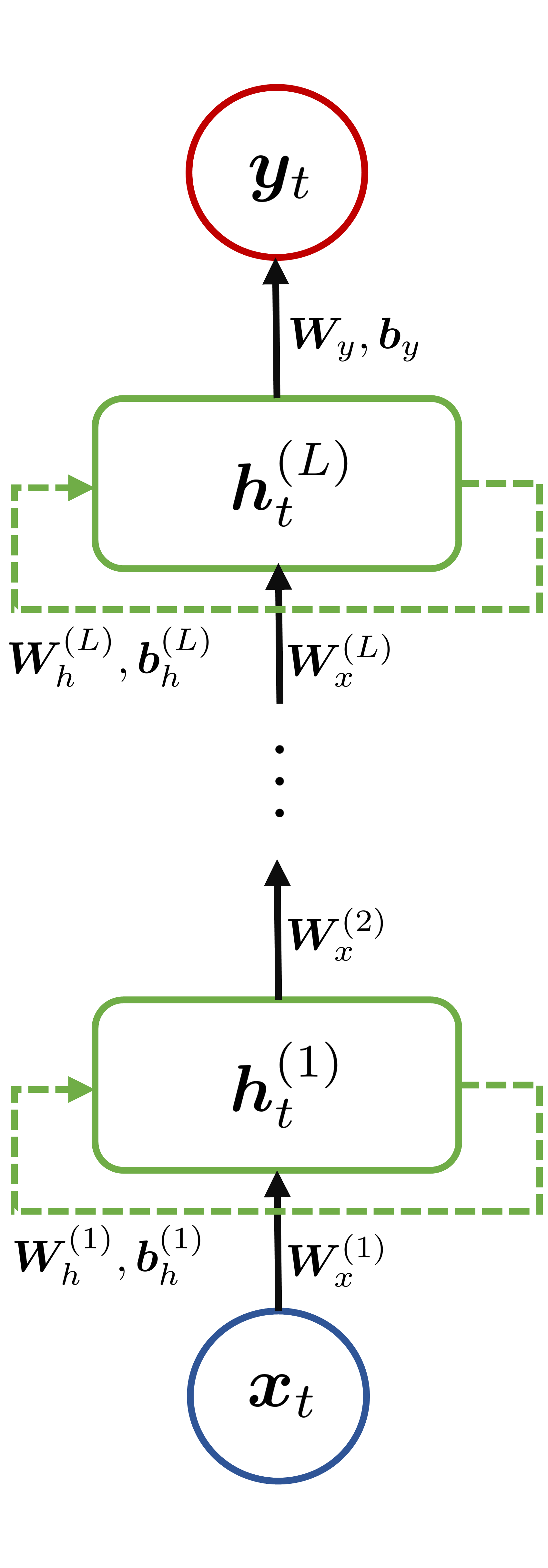}
        \caption{RNN}
        \label{fig:rnn1}
    \end{subfigure}
    \hfill
    \begin{subfigure}[t]{0.7255\linewidth}
        \centering
        \includegraphics[width=0.9\linewidth]{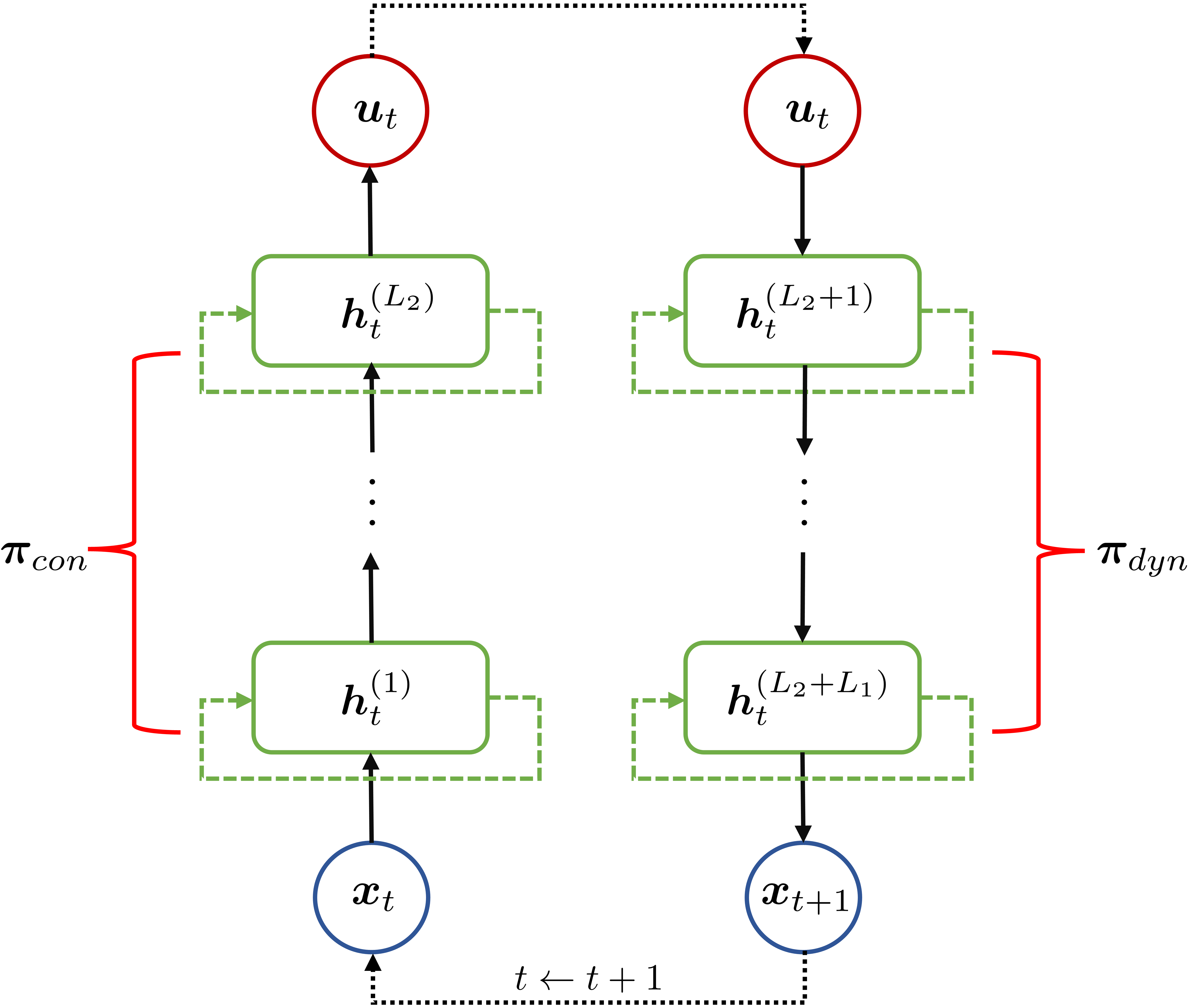}
        \caption{Closed-loop RNN system}
        \label{fig:rnn2}
    \end{subfigure}
    \caption{(a) Diagram of an isolated $L$-layered  RNN. (b) Diagram of a closed-loop RNN system comprising of an RNN plant and an RNN controller.}
    \label{fig:rnn}
\end{figure}

\subsection{Problem Statement}
Consider the following discrete-time nonlinear plant:
\begin{align}
    \x_{t+1} & = f(\x_t,\u_t), \label{equ:state-space}
\end{align}
and a controller:
\begin{align}
    \bm{u}_t &= q(\x_t,\u_{t-1}),\label{equ:control}
\end{align}
where $\x$ denotes the state and $\u$ the control input. We assume that the plant dynamics \eqref{equ:state-space} are approximated by an $L_1$-layered RNN, denoted as $\x_{t+1} = \mt{\pi}_{dyn}(\u_t)$, and the controller \eqref{equ:control} is approximated by an $L_2$-layered RNN, denoted as $\u_t = \mt{\pi}_{con}(\x_t)$ \cite{d2023incremental,bonassi2022recurrent,la2025regional}.  The two RNNs $\mt{\pi}_{dyn}$ and $\mt{\pi}_{con}$ can be connected by relating their input and output. The resulting compositional system forms an $L$-layered RNN, denoted by $\mt{\pi}$, where $L = L_1+L_2$ (see Figure \ref{fig:rnn} (b)). 
Compared to an isolated RNN \eqref{equ:RNN} whose inputs and outputs may be externally specified and unconnected, the RNN $\mt{\pi}$ takes the output from the last timestep as the current timestep input, i.e., 
\begin{equation}\label{equ:RNN2}
\bm{x}_{t+1}  = \bm{y}_{t} = \sigma_y(\bm W_y \bm{h}_t^{(L)} + \bm{v}_y ).
\end{equation}
We call the resulting feedback control system $\mt{\pi}$ the  \emph{closed-loop RNN system}, and  denote it as 
\begin{align}\label{equ:sys}
    \x_{t+1} = \mt{\pi}(\x_t),
\end{align}
where $\x_t \in \X \subset \mathbb{R}^n$ is the state of the closed-loop system at time $t$ and $\mathcal{X}$ is the state domain. 
We use the notation $\mt{\pi}_h(\x_1,t,\ell)$ to denote the hidden state value in the $\ell$-th layer of $\mt{\pi}$ and at the $t$-th step, for a given initial state $\x_1 \in \X$, by \eqref{equ:RNN-1}, \eqref{equ:RNN-2} and \eqref{equ:RNN2}. Specifically, $\mt{\pi}_h(\x_1,t,\ell)\triangleq \h^{(\ell)}_t \in \mathbb{R}^{n_\ell}$ from $\h_1^{(0)} = \x_1, \h_0^{(r)} = \bm 0, \bm{h}_k^{(r)} = \sigma_h^{(r)}(\bm W_h^{(r)} \bm{h}_{k-1}^{(r)} + \bm W_x^{(r)} \bm{h}_{k}^{(r-1)} + \bm{v}_h^{(r)} ), \h^{(0)}_{k+1} = \sigma_y(\bm W_y \bm{h}_k^{(L)} + \bm{v}_y )$, for $k \in [t-1]$ with $r \in [L]$, and for  $k = t$ with  $r \in [\ell]$.

Given an initial set $\X_1 \subseteq \X$, the (exact) \emph{forward reachable set} (FRS) of the closed-loop RNN \eqref{equ:sys} at step $t$ is defined as $\mathcal{R}_t(\X_1) = \{\x_t \in \mathbb{R}^n\;|\;\x_1\in \X_1, \x_{k+1} = \mt{\pi}(\x_k), k \in [t-1]\}$. Similarly, given a target set $\mathcal{T} \subseteq \X$, the (exact) \emph{backward reachable set} (BRS) of the closed-loop RNN \eqref{equ:sys} at step $t$ is defined as $\mathcal{P}_t(\mathcal{T}) = \{\x_1 \in \mathcal{X}\;|\;\x_t\in \mathcal{T}, \x_{k+1} = \mt{\pi}(\x_k), k \in [t-1]\}$. 



In this work, we aim to address the following problem: \emph{Given an initial set $\X_1 \subseteq \X$, where both $\X_1$ and $\X$ are  represented as HZs, compute the $t$-step exact and over-approximated FRSs for the system \eqref{equ:sys}, expressed in the form of HZs. Similarly, given an target set $\mathcal{T} \subseteq \X$, where $\mathcal{T}$ is an HZ, compute the $t$-step exact and over-approximated BRSs for the system \eqref{equ:sys}, in the form of HZs.}

\section{Main Results}\label{sec:RNN}

In this section, we first introduce the concept of state-pair sets and present an approach to construct HZ state-pair sets that represent the relationship between the initial state $\x_1$ and the $t$-step state $\x_t$ for $t\in \mathbb{Z}_{>1}$. Building on this formulation, we then compute the FRSs and BRSs of the closed-loop RNN \eqref{equ:sys}, and derive a sufficient condition for safety verification.

\subsection{State-Pair Set}

For FNNs, such as multi-layer perceptrons and convolutional neural networks, it has been shown that their input-output relationships can be represented as HZs by propagating input sets layer-by-layer through the network  \cite{zhang2024reachability,zhang2025efficient}. However, as defined in \eqref{equ:RNN}, RNNs exhibit interconnections between two hidden states across both layers and time steps (i.e., $\h^{(\ell)}_{t-1}$ and $\h^{(\ell-1)}_{t}$). These temporal and structural dependencies prevent direct application of FNN-based reachability methods to RNNs. To address this challenge, we introduce two new sets that explicitly encode these interconnections.

We first define the \emph{state-pair set} of the closed-loop RNN system \eqref{equ:sys} with state domain set $\X$.
\begin{definition}
    The state-pair set $\mathcal{S}_{x}(\X,t)\subset \mathbb{R}^{2n}$ at timestep $t\in\mathbb{Z}_{\geq 2}$ with respect to the state domain $\X$ is 
    \begin{align*}
         \mathcal{S}_{x}(\X,t) = \{(\x_1,\x_t)\;|\; & \x_1\in \X, \x_{k+1} = \mt{\pi}(\x_k), k \in [t-1]\}.
    \end{align*}
    \vskip 1mm
\end{definition}

From the definition above, it is evident that the state-pair set preserves the relationship between the initial state and the $t$-step state of the closed-loop RNN. 

Similarly, we define the \emph{hidden-state-pair set} for the closed-loop RNN system \eqref{equ:sys}, which encodes all possible transitions between the connected hidden states.

\begin{definition}
    The hidden-state-pair set $\mathcal{S}_{h}(\X,t,\ell)\subset \mathbb{R}^{n_{\ell-1}+n_\ell}$ with respect to the state domain  $\X$ is 
    \begin{align*}
         \mathcal{S}_{h}(\X,t,\ell) = \{ & (\h^{(\ell)}_{t-1},\h^{(\ell-1)}_{t}) \;|\;  \h^{(\ell)}_{t-1} = \mt{\pi}_h(\x_1,t-1,\ell), \\ & \h^{(\ell-1)}_{t} = \mt{\pi}_h(\x_1,t,\ell-1), \x_1 \in \X\}.
    \end{align*}
    \vskip 1mm
\end{definition}

Next, we define a new operation on two HZs with special structures, which will be used to construct both the state-pair set and the hidden-state-pair set.

\begin{definition}\label{def:cp}
    Let $\mathcal{Z}=\hz{z} \subset \mathbb{R}^{n_x}$ and $\mathcal{Y}=\hz{y} \subset \mathbb{R}^{n_y}$ be two HZs. 
    Suppose that matrices $\A^c_y$, $\A^b_y$, and $\bb_y$ can be partitioned as 
    \begin{align}\label{equ:cp}
        \A^c_y = \mat{\A^c_z & \bm 0\\ \multicolumn{2}{c}{\bar{\A}^c_y}},\;\A^b_y = \mat{\A^b_z & \bm 0\\ \multicolumn{2}{c}{\bar{\A}^b_y}},\;\bb_y = \mat{\bb_z\\\bar{\bb}_y}.
    \end{align}
    Then we define the \emph{constrained product} of $\mathcal{Z}$ and $\mathcal{Y}$ as
    \begin{align}
        \mathcal{C}(\mathcal{Z},\mathcal{Y}) \triangleq \langle \mat{\G^c_z & \bm 0\\ \multicolumn{2}{c}{{\G}^c_y}}, \mat{\G^b_z & \bm 0\\ \multicolumn{2}{c}{{\G}^b_y}}, \mat{\cc_z\\ \cc_y}, \A^c_y, \A^b_y, \bb_y \rangle.
    \end{align}
    \vskip 2mm
\end{definition}

The definition above shows that the constrained product operation combines the equality constraints of two HZs, resulting in an HZ that is always a subset of their Cartesian product, i.e., $\mathcal{C}(\mathcal{Z},\mathcal{Y}) \subseteq \mathcal{Z} \times \mathcal{Y}$. We call two HZs well-defined for constrained product if they satisfy the condition \eqref{equ:cp} in Definition \ref{def:cp}. 



Since this work considers ReLU activation functions, we utilize the following lemma that states the graph of the ReLU activation function can be exactly represented as an HZ.

\begin{lemma}\cite[Lemma 2]{zhang2024reachability}\label{lem:relu}
    Given an HZ $\mathcal{Z}\subset \mathbb{R}^{n_{\ell}}$ and its interval hull $\mathcal{I}\triangleq \li\mt{\alpha},\mt{\beta}\ri = \itl(\mathcal{Z})$, then the graph of the vector-valued ReLU activation function $\sigma: \mathbb{R}^{n_{\ell}}\rightarrow \mathbb{R}^{n_{\ell}}$ over $\mathcal{Z}$ can be exactly represented as 
\begin{equation}\label{equ:g_phi}
    \mathcal{G}_{\sigma}(\mathcal{Z}) = (\bm{P}\cdot \mathcal{G}_{ReLU}(\mathcal{I}))\cap_{[\bm{I}\; \bm{0}]} \mathcal{Z},
\end{equation}
where $\bm{P}=[\bm{e}_{2}\; \bm{e}_{4}\;\cdots\; \bm{e}_{2n_{\ell}} \; \bm{e}_{1}\; \bm{e}_{3}\;\cdots\; \bm{e}_{2n_{\ell}-1}]^T\in \mathbb{R}^{2n_{\ell}\times 2n_{\ell}}$ is a permutation matrix, $\mathcal{G}_{ReLU}(\mathcal{I}) = \mathcal{G}_{ReLU}(\li\alpha_1,\beta_1 \ri) \times \cdots \times \mathcal{G}_{ReLU}(\li\alpha_{n_\ell},\beta_{n_\ell}\ri)$ and
\begin{align*}
&{\mathcal{G}}_{ReLU}(\li\alpha_i,\beta_i\ri) =\nonumber\\
&\left\{ \begin{array}{l}
\mathcal{H}_+\triangleq \lrangle{\mat{\frac{\beta_i-\alpha_i}{2} \\ \frac{\beta_i-\alpha_i}{2}},\emptyset,\mat{\frac{\beta_i+\alpha_i}{2} \\ \frac{\beta_i+\alpha_i}{2}},\emptyset,\emptyset,\emptyset}, \;\; \mbox{if} \; 0\leq  \alpha_i,\\ 
\mathcal{H}_-\triangleq \lrangle{\mat{\frac{\beta_i-\alpha_i}{2} \\ 0}, \emptyset,\mat{\frac{\beta_i+\alpha_i}{2} \\ 0}, \emptyset,\emptyset,\emptyset}, \;\; \mbox{if}\; \beta_i \leq 0,\\
\mathcal{H}_\pm\triangleq = \hz{\pm}, \;\; \mbox{if}\; \alpha_i <0<\beta_i,
\end{array}\right.
\end{align*} 
for $i\in[n_\ell]$, and the expressions of $\G^c_\pm, \G^b_\pm, \cc_\pm, \allowbreak\A^c_\pm,\allowbreak \A^b_\pm, \allowbreak\bb_\pm$ can be found in \cite[eq. (3)]{zhang2023backward}. 
\end{lemma}


From the structure of the closed-loop RNN system shown in \eqref{equ:RNN-1}, \eqref{equ:RNN-2}, and \eqref{equ:RNN2}, we can observe that each hidden-layer update depends on both the hidden state from the previous layer and the hidden state from the previous timestep. Consequently, to construct HZ-represented state-pair sets for the closed-loop RNN system over the entire domain, we propagate the state-domain set $\X$ as an HZ through each ReLU-activated layer using Lemma \ref{lem:relu}, and connect the input of each hidden layer with the hidden states from the preceding layer and the previous timestep via the linear map on the hidden-state-pair set. The detailed procedure is summarized in Algorithm \ref{alg:1}. 

In Algorithm \ref{alg:1}, for the first time step ($t=1$), since all initial hidden states are assumed to be zero, the HZ-represented hidden state sets $\mathcal{H}_1^{(\ell)}$ can be obtained by propagating $\X$ through each of the $L$ hidden layers and the output layer, similar to the case of FNNs in \cite[Algorithm~2]{zhang2024reachability}. For a general step $t \geq 2$, suppose the hidden state sets of the $(\ell-1)$-th layer, $\mathcal{H}^{(\ell-1)}_t$, have already been computed. Then, in Line~7, the hidden-state-pair set can be represented as an HZ by taking the constrained product of $\mathcal{H}^{(\ell-1)}_t$ and $\mathcal{H}^{(\ell)}_{t-1}$. Using this hidden-state-pair set, the input to the $\ell$-th hidden layer is computed through a linear mapping of HZs in Line~8. The remaining procedure is the same as the first time step. By repeating this process across all layers and all time steps, we obtain the hidden state sets for all $t \in [T]$ and $\ell \in \{0,1,\dots,L\}$. Finally, since the output of the RNN at step $t$ becomes the state of the closed-loop RNN system at step $t+1$ (see \eqref{equ:RNN2}), in Line~16 we construct the state-pair sets by computing the constrained product of the input set $\mathcal{H}_1^{(0)}$ and output set $\mathcal{H}_{t+1}^{(0)}$.

The next lemma shows that the hidden-state-pair sets obtained from the constrained product in Line 7 of Algorithm \ref{alg:1} are valid.

\begin{lemma}\label{lem:pair-set}
    Consider the closed-loop RNN system \eqref{equ:sys} and any given HZ $\X \subset \mathbb{R}^n$, let $\{\mathcal{H}^{(\ell)}_t\}_{t\in\{0,\dots,T-1\},\ell\in\{0,\dots,L\}}$ be the HZ-represented hidden state sets, computed as in Algorithm \ref{alg:1}, i.e., $\mathcal{H}^{(\ell)}_t = \{\h^{(\ell)}_{t} \;|\;  \h^{(\ell)}_{t} = \mt{\pi}_h(\x_1,t,\ell), \x_1 \in \X\}$. Then, for $t\in[T-1]$ and $\ell\in[L]$,  the constrained product of $\mathcal{H}_{t-1}^{(\ell)}$ and $\mathcal{H}_t^{(\ell-1)}$ is an exact HZ representation for the hidden-state-pair set, i.e.,
    $
        \mathcal{S}_{h}(\X,t,\ell) = \mathcal{C}(\mathcal{H}_{t-1}^{(\ell)},\mathcal{H}_t^{(\ell-1)}).
    $
\end{lemma}

\begin{algorithm}[!t]
\SetNoFillComment
\caption{Exact state-pair set computation via HZs}\label{alg:1}
\KwIn{HZ state domain set $\X$, weight matrices $\{\W_{h}^{(\ell)}\}_{\ell=1}^L$, $\{\W_{x}^{(\ell)}\}_{\ell=1}^L$ and $\W_{y}$, bias vectors $\{\v_{h}^{(\ell)}\}_{\ell=1}^L$ and $\v_y$, number of steps $T>1$}
\KwOut{a sequence of state-pair sets $\{\mathcal{S}_{x,t}\}^T_{t=2}$}
$\mathcal{H}_1^{(0)} \leftarrow \X$;\\
\For{$t \in \{1,\dots,T-1$\}}{
\For{$\ell \in \{1,\dots,L$\}}{
\uIf{$t = 1$}{$\mathcal{Z}^{(\ell)}_{h,1} \leftarrow \bm{W}_{x}^{(\ell)}\mathcal{H}_1^{(\ell-1)}+\bm{v}_{h}^{(\ell)}$;\\}
\uElse{
$\mathcal{S}_{h,t}^{(\ell)}\leftarrow\mathcal{C}(\mathcal{H}_{t-1}^{(\ell)},\mathcal{H}_t^{(\ell-1)})$\tcp*{Hidden-state-pair set}
$\mathcal{Z}^{(\ell)}_{h,t} \leftarrow [\bm{W}_{h}^{(\ell)}\; \bm{W}_{x}^{(\ell)}] \cdot \mathcal{S}_{h,t}^{(\ell)}+\bm{v}_{h}^{(\ell)}$;\\}
$\mathcal{I}_{h,t}^{(\ell)} = \itl(\mathcal{Z}^{(\ell)}_{h,t})$\tcp*{Using CROWN}
$\mathcal{G}_{h,t}^{(\ell)}$ $\leftarrow$ $(\bm{P}\cdot\mathcal{G}_{ReLU}(\mathcal{I}_{h,t}^{(\ell)}))\cap_{[\bm{I}\;\bm{0}]} \mathcal{Z}^{(\ell)}_{h,t}$\tcp*{Hidden layer ReLU graph}
$\mathcal{H}_{t}^{(\ell)}$ $\leftarrow$ $[\bm{0}\;\bm{I}] \cdot \mathcal{G}_{h,t}^{(\ell)}$\tcp*{Hidden state set}
}
$\mathcal{Z}_{y,t} \leftarrow \bm{W}_{y}\mathcal{H}_{t}^{(L)} +\bm{v}_{y}$;\\
$\mathcal{I}_{y,t} = \itl(\mathcal{Z}_{y,t})$\tcp*{Using CROWN}
$\mathcal{G}_{y,t}$ $\leftarrow$ $(\bm{P}\cdot\mathcal{G}_{ReLU}(\mathcal{I}_{y,t}))\cap_{[\bm{I}\;\bm{0}]} \mathcal{Z}_{y,t}$\tcp*{Output layer ReLU graph}
$\mathcal{H}_{t+1}^{(0)}$ $\leftarrow$ $[\bm{0}\;\bm{I}] \cdot \mathcal{G}_{y,t}$\tcp*{Output set}
$\mathcal{S}_{x,t+1}$ $\leftarrow$ $\mathcal{C}(\mathcal{H}_1^{(0)},\mathcal{H}_{t+1}^{(0)})$\tcp*{State-pair set}
}
\Return{$\{\mathcal{S}_{x,t}\}^T_{t=2}$}
\end{algorithm}

\begin{proof}
    By the properties of linear mappings and generalized intersections of HZs (Lemma~\ref{lemma:inter}), the equality constraints encoded in $\mathcal{H}_{t-1}^{(\ell)}$ are contained in those of $\mathcal{H}_{t}^{(\ell-1)}$. Hence, the constrained product $\mathcal{C}(\mathcal{H}_{t-1}^{(\ell)},\mathcal{H}_{t}^{(\ell-1)})$ is well defined. By relating the common equality constraints, we enforce that the pairs of hidden states in the constrained product are generated by the same initial state, i.e., $\mathcal{C}(\mathcal{H}_{t-1}^{(\ell)},\mathcal{H}_t^{(\ell-1)}) = \{(\h^{(\ell)}_{t-1},\h^{(\ell-1)}_{t}) \;|\;  \h^{(\ell)}_{t-1} = \mt{\pi}_h(\x_1,t-1,\ell), \h^{(\ell-1)}_{t} = \mt{\pi}_h(\x_1',t,\ell-1), \x_1 \in \X, \x_1'\in \X, \x_1 = \x_1'\} = \mathcal{S}_{h}(\X,t,\ell)$, which completes the proof.
\end{proof}

\begin{remark}
Lines 9 and 13 of Algorithm \ref{alg:1} involve computing the interval hulls of the HZ sets $\mathcal{Z}^{(\ell)}_{h,t}$ and $\mathcal{Z}_{y,t}$. The exact interval hull can be obtained by solving $2n_\ell$ MILPs -- one minimization and one maximization per coordinate -- to obtain tight lower and upper bounds (see \cite[Proposition~3.2.10]{bird2022hybrid}). However, this approach quickly becomes intractable as $n_\ell$ increases. To address this, in our implementation we approximate these per-layer interval bounds using CROWN  with GPU acceleration \cite{zhang2018efficient,wang2021betaetal}. Notably, CROWN provides sound interval enclosures that substantially enhance scalability while preserving exactness in the subsequent state-pair set construction.
\end{remark}

The following proposition shows that the computation of state-pair sets in Algorithm \ref{alg:1} is sound.

\begin{proposition}\label{prop1:state-pair}
Consider the closed-loop RNN system \eqref{equ:sys} and any given HZ $\X \subset \mathbb{R}^n$, the output of Algorithm \ref{alg:1}, $\{\mathcal{S}_{x,t}\}^T_{t=2}$, is a sequence of HZs that can exactly represent the state-pair sets of \eqref{equ:sys} over the domain $\X$.
\end{proposition}
\begin{proof}
    By construction in Algorithm \ref{alg:1}, we know the equality constraints in the domain set $\X$ are included in the state set $\mathcal{H}^{(0)}_{t+1} = \{\x_{t+1}\;|\; \x_{k+1} = \mt{\pi}(\x_k),k\in[t],\x_1\in\X\}$. Therefore, $\mathcal{H}_1^{(0)}$ and $\mathcal{H}^{(0)}_{t+1}$ are well-defined for constrained product since $\mathcal{H}_1^{(0)}= \X$. For $t= 1,\dots,T-1$, the output of Algorithm \ref{alg:1}, $\mathcal{S}_{x,t+1}$, stacks the initial state and the $t+1$ step state as $ \mathcal{S}_{x,t+1}= \mathcal{C}(\mathcal{H}_1^{(0)},\mathcal{H}_{t+1}^{(0)}) = \mathcal{C}(\X,\mathcal{H}_{t+1}^{(0)}) = \{(\x,\x_{t+1})\;|\;\x\in\X,\x_{k+1} = \mt{\pi}(\x_k),k\in[t],\x_1\in\X,\x = \x_1)\} = \mathcal{S}_x(\X,t)$, which is an exact HZ representation of the state-pair set of \eqref{equ:sys} over the domain $\X$.
\end{proof}

\begin{remark}
Denote $n_{g,x}$, $n_{b,x}$, and $n_{c,x}$ as the numbers of continuous generators, binary generators, and equality constraints of the HZ $\mathcal{X}$, respectively. Let $N_t$ be the total number of \emph{unstable} ReLU activations in the RNN at step $t$ (i.e., those whose input interval $[\alpha,\beta]$ satisfies $\alpha<0<\beta$). Then, the set complexity of $\mathcal{S}_{x,t}$ from Algorithm \ref{alg:1} is: 
\begin{align}
    n_{g,s}  \!=\! n_{g,x} \!+ \!4 N_t, n_{b,s} \!=\! n_{b,x}\! +\! N_t,n_{c,s} \!=\! n_{c,x} \!+\! 3 N_t.\label{equ:nsindex}
\end{align}
\end{remark}

\subsection{Exact Forward and Backward Reachability}

In this subsection, we compute the exact FRSs and BRSs for the closed-loop RNN system \eqref{equ:sys}.

\begin{theorem}\label{thm1}
Consider the closed-loop RNN system \eqref{equ:sys} and any given HZ $\mathcal{X} \subset\mathbb{R}^n$. Let $\{\mathcal{S}_{x,t}\}^T_{t=2}$ be the state-pair sets over the domain $\mathcal{X}$ computed using Algorithm \ref{alg:1}. 

(i) For any initial set $\mathcal{X}_1 \subseteq \mathcal{X} $ represented by an HZ, the $t$-step FRS of \eqref{equ:sys}, where $t= 2,\dots,T$, can be represented by the following HZ: 
$$
{\mathcal{R}}_t(\mathcal{X}_1) = [\bm 0_{n\times n} \; \bm I_n]\cdot (\mathcal{S}_{x,t}\cap_{[\bm I_n \; \bm 0_{n\times n}]} \X_1).
$$

(ii) For any target set $\mathcal{T}$ represented by an HZ, the $t$-step BRS of \eqref{equ:sys} in the domain $\mathcal{X}$, where $t= 1,\dots,T$, can be represented by the following HZ: 
$$
{\mathcal{P}}_t(\mathcal{T}) = [\bm I_n \; \bm 0_{n\times n}]\cdot(\mathcal{S}_{x,t} \cap_{[\bm 0_{n\times n} \; \bm I_n]}\mathcal{T}).
$$
\end{theorem}
\begin{proof}
    (i) Since $\mathcal{S}_{x,t} =  \mathcal{S}_x(\X,t) = \{(\x_1,\x_{t})\;|\;\x_{k+1} = \mt{\pi}(\x_k),k\in[t-1],\x_1\in\X)\}$, we have $ \mathcal{S}_x(\X_1,t) = \mathcal{S}_{x,t}\cap_{[\bm I_n \; \bm 0_{n\times n}]} \X_1$. Therefore, $\mathcal{R}_t(\X_1) = \{\x_t \;|\;\x_1\in \X_1, \x_{k+1} = \mt{\pi}(\x_k), k \in [t-1]\} = [\bm 0_{n\times n} \; \bm I_n]\cdot \mathcal{S}_x(\X_1,t) = [\bm 0_{n\times n} \; \bm I_n]\cdot (\mathcal{S}_{x,t}\cap_{[\bm I_n \; \bm 0_{n\times n}]} \X_1)$.
    
    (ii) Since $\mathcal{S}_{x,t} =  \mathcal{S}_x(\X,t) = \{(\x_1,\x_{t})\;|\;\x_{k+1} = \mt{\pi}(\x_k),k\in[t-1],\x_1\in\X)\}$, we have $\mathcal{P}_t(\mathcal{T}) = \{\x_1 \in \mathcal{X}\;|\;\x_t\in \mathcal{T}, \x_{k+1} = \mt{\pi}(\x_k), k \in [t-1]\} = [\bm I_n \; \bm 0_{n\times n}]\cdot(\mathcal{S}_{x,t} \cap_{[\bm 0_{n\times n} \; \bm I_n]}\mathcal{T})$.
\end{proof}

Prior work on the  verification of \emph{isolated} RNNs has primarily focused on forward reachability using convex set abstractions or transformations to FNNs. For example, unrolling-based methods proposed in \cite{akintunde2019verification}, expand an RNN into a large equivalent FNN over a fixed horizon, while invariant inference methods in \cite{jacoby2020verifying} construct FNN over-approximations, allowing well-established FNN verification techniques to be applied. To avoid unrolling, set-propagation methods have also been studied by using different set representations including star-set, sparse star sets.  
However, these efforts address only forward reachability and do not consider backward reachability or the feedback coupling present in closed-loop settings. For RNNs embedded in control loops, recent literature has instead emphasized stability guarantees using LMI-based analyses (e.g., \cite{revay2021convex,d2023incremental,la2025regional}). To the best of our knowledge, Theorem~\ref{thm1} provides the first result on closed-loop RNN reachability that uifies the computation of both forward and backward reachable sets.

\begin{remark}
Assume the initial set $\mathcal{X}_1$ (resp., target set $\mathcal{T}$) has $(n_{g,1}, n_{b,1}, n_{c,1})$ (resp., $(n_{g,\tau}, n_{b,\tau}, n_{c,\tau})$) continuous generators, binary generators, and equality constraints, respectively. Then, the set complexity of the FRS $\mathcal{R}_t(\mathcal{X}_1)$ is given by 
$n_{g,r} = n_{g,s}+n_{g,1}$, $n_{b,r} = n_{b,s}+n_{b,1}$, $n_{c,r} = n_{c,s}+n+n_{c,1}$,
and that of the BRS $\mathcal{P}_t(\mathcal{T})$ is given by 
$n_{g,p} = n_{g,s}+n_{g,\tau}$, $n_{b,p} = n_{b,s}+n_{b,\tau}$, $n_{c,p} = n_{c,s}+n + n_{c,\tau}$,
where $n_{g,s}$, $n_{b,s}$ and $n_{c,s}$ are given in \eqref{equ:nsindex}.
\end{remark}

\subsection{Over-approximated Forward and Backward Reachability}\label{sec:NN}

From the exact reachability analysis in the previous subsection, the complexity of the HZ-represented reachable sets grows proportionally with the total number of unstable ReLU activations $N_t$. The quantity $N_t$ itself depends on the size of the initial set, the size of the RNN, and the number of time steps $t$ in the closed-loop evolution \eqref{equ:sys}. To improve the scalability of the proposed HZ-based approach, we now present an over-approximation procedure that enforces a desired complexity limit.

Based on the properties of HZ operations involved in Algorithm \ref{alg:1}, it can be observed that the only procedures that increase the set complexity of the HZ state-pair sets are the steps (i.e., Line 10 and Line 14) that involve the generalized intersection as defined in \eqref{equ:g_phi} for the computation of the vector-valued graph set of ReLU activations. Thus, reducing the set complexity amounts to replacing the exact ReLU graph by a relaxed representation. Specifically, for each unstable ReLU, we over-approximate its exact graph $\mathcal{H}_\pm$ by the standard triangle (convex-hull) relaxation $\mathcal{H}_\triangle$, as shown in Figure \ref{fig:ReLU}, i.e, $\mathcal{H}_\pm \subset \mathcal{H}_\triangle = \langle [\G^c_\pm\; \G^b_\pm],\emptyset,\cc_\pm,[\A^c_\pm\; \A^b_\pm],\emptyset,\bb_\pm\rangle$. Here $\mathcal{H}_\triangle$ is an HZ with no binary generators (a degenerate HZ), and hence reduces to a constrained zonotope~\cite{scott2016constrained}. This substitution prevents the introduction of additional binary generators at the ReLU graph stage and thereby controls the growth of set complexity.

\begin{figure}[!t]
    \centering
    \subfloat[]{\includegraphics[width=0.43\linewidth]{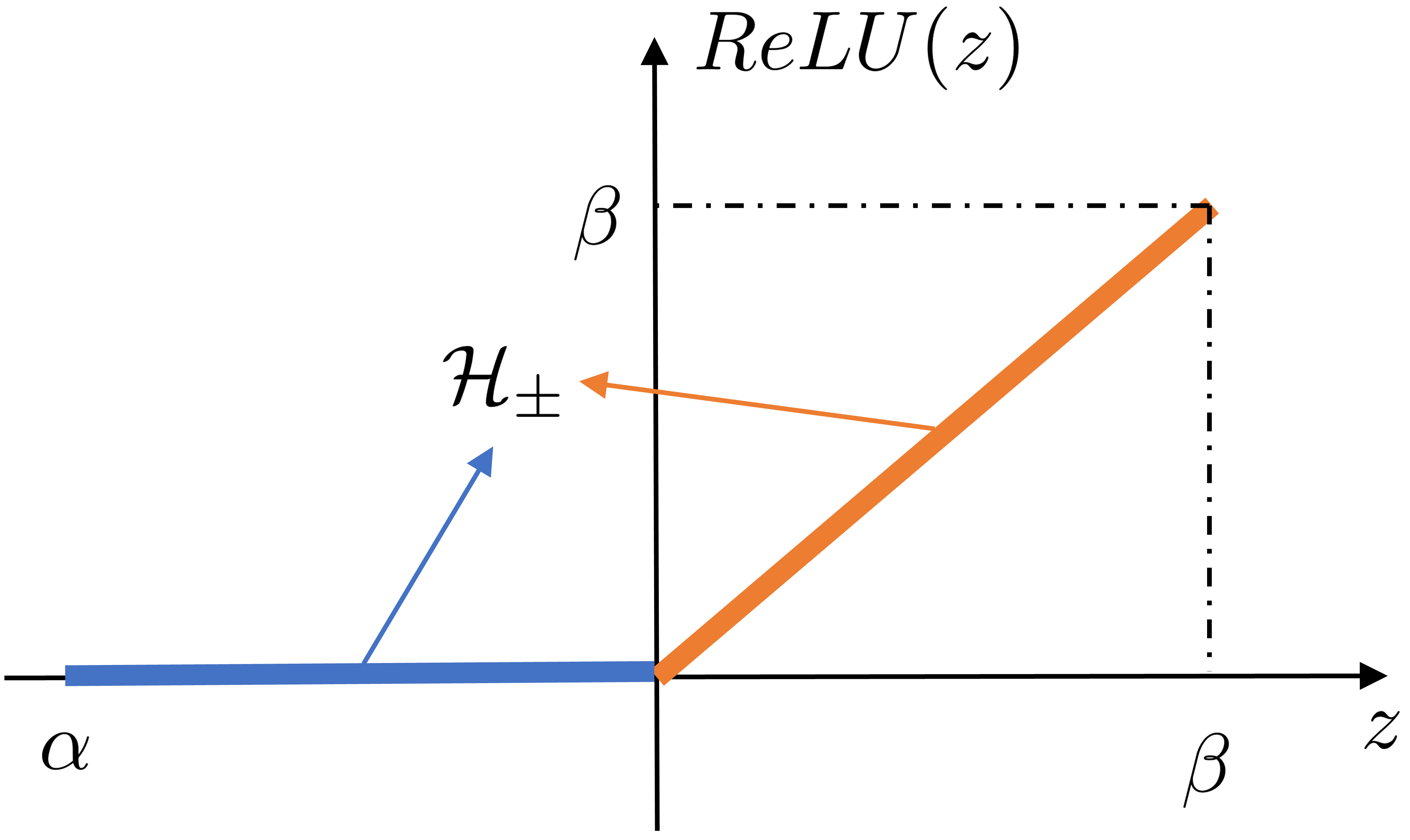}}
    \hspace{1em}
    \subfloat[]{\includegraphics[width=0.43\linewidth]{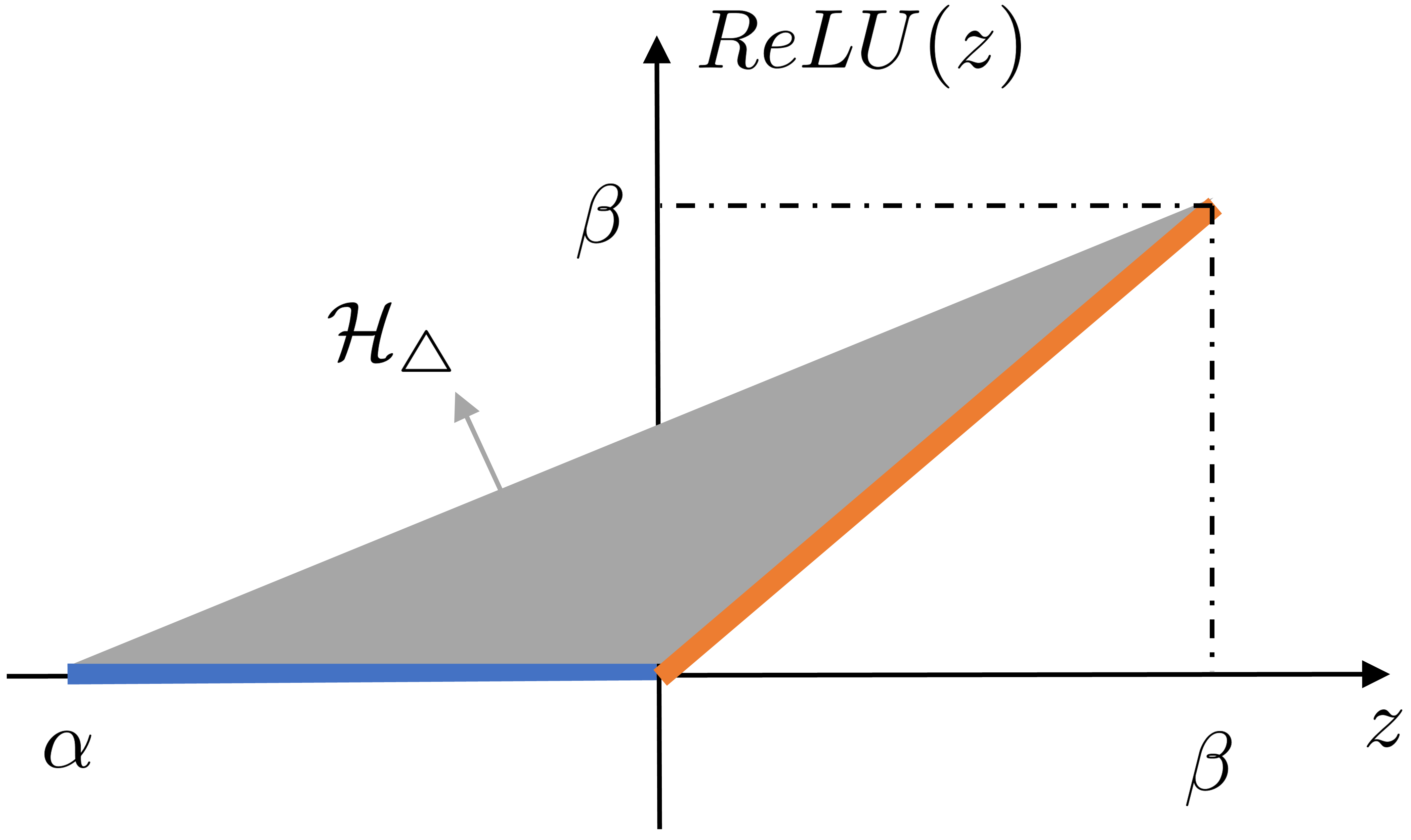}}
    \caption{The graph of a ReLU function over $\li\alpha,\beta\ri$ can be represented as an HZ. (a) Graph $\mathcal{H}_{\pm}$ with $ \alpha < 0 < \beta$. (b) The triangle over-approximation of the graph $\mathcal{H}_{\triangle} \supset \mathcal{H}_\pm$. The area of $\mathcal{H}_{\triangle}$ is $-\frac{\alpha\cdot\beta}{2}$. }
    \label{fig:ReLU}
\end{figure}

Although we can apply this triangle relaxation to all the unstable ReLU activations, this tends to make the subsequent state-pair set and reachable set computation overly conservative, due to the convex-relaxation barrier  \cite{salman2019convex}. 
To mitigate conservatism while controlling complexity, we selectively apply the triangle relaxation only to those unstable ReLUs that are expected to have limited impact on the tightness of the reachable sets. Given precomputed input bounds $\li\alpha,\beta\ri$ for each unstable ReLU (with $\alpha<0<\beta$), we assign a heuristic “triangle-area” score
\begin{equation}\label{equ:score}
score = -\frac{\alpha\cdot\beta}{2},    
\end{equation}
which is the area of the triangular gap of the relaxation. We then rank the unstable ReLUs by $score$ in descending order and introduce a binary limit parameter $N_b\in\mathbb{Z}_{\ge 0}$. The top $N_b$ unstable ReLUs with the highest scores are labeled $0$ and represented with the exact ReLU graph, while the remaining unstable ReLUs are labeled $1$ and approximated by their triangle relaxations, i.e.,
\begin{align*}
        \tilde{\mathcal{G}}_{ReLU}(\li \alpha,\beta \ri,label) \!=\! \left\{ \!\!\!\begin{array}{l}
\mathcal{H}_\triangle, \;\mbox{if}\; (\alpha \!<\!0\!<\!\beta) \wedge (label = 1), \\ 
\mathcal{G}_{ReLU}(\li \alpha,\beta \ri), \;\;\;\;\;\mbox{otherwise}. \\
\end{array}\right. \!\!\!\!\!
\end{align*}

\begin{algorithm}[!t]
\SetNoFillComment
\caption{Over-approximated state-pair set computation via HZs}\label{alg:2}
\KwIn{HZ state domain set $\X$, RNN $\mt{\pi}$ with weight matrices $\{\W_{h}^{(\ell)}\}_{\ell=1}^L$, $\{\W_{x}^{(\ell)}\}_{\ell=1}^L$ and $\W_{y}$, bias vectors $\{\v_{h}^{(\ell)}\}_{\ell=1}^L$ and $\v_y$, number of steps $T\in\mathbb{Z}_{>1}$, binary limit parameter $N_b\in\mathbb{Z}_{\geq 0}$}
\KwOut{over-approximated state-pair sets $\{\tilde{\mathcal{S}}_{x,t}\}^T_{t=2}$}
$\{\mathcal{I}_{h,t}^{(\ell)},\mathcal{I}_{y,t}\}_{t\in[T-1],\ell\in[L]} \leftarrow CROWN(\mt{\pi},\X,T)$\tcp*{Precomputation of interval bounds}
$\{{label}_{h,t}^{(\ell)},{label}_{y,t}\}_{t\in[T-1],\ell\in[L]}\!\leftarrow\! rank(\{\mathcal{I}_{h,t}^{(\ell)},\mathcal{I}_{y,t}\},N_b)$\tcp*{Rank and label ReLU activations}
$\tilde{\mathcal{H}}_1^{(0)} \leftarrow \X$;\\
\For{$t \in \{1,\dots,T-1$\}}{
\For{$\ell \in \{1,\dots,L$\}}{
\uIf{$t = 1$}{$\tilde{\mathcal{Z}}^{(\ell)}_{h,1} \leftarrow \bm{W}_{x}^{(\ell)}\tilde{\mathcal{H}}_1^{(\ell-1)}+\bm{v}_{h}^{(\ell)}$;\\}
\uElse{
$\tilde{\mathcal{S}}_{h,t}^{(\ell)}\leftarrow\mathcal{C}(\tilde{\mathcal{H}}_{t-1}^{(\ell)},\tilde{\mathcal{H}}_t^{(\ell-1)})$\tcp*{Hidden-state-pair set}
$\tilde{\mathcal{Z}}^{(\ell)}_{h,t} \leftarrow [\bm{W}_{h}^{(\ell)}\; \bm{W}_{x}^{(\ell)}] \cdot \tilde{\mathcal{S}}_{h,t}^{(\ell)}+\bm{v}_{h}^{(\ell)}$;\\}
$\tilde{\mathcal{G}}_{h,t}^{(\ell)}$ $\leftarrow$ $(\bm{P}\cdot\tilde{\mathcal{G}}_{ReLU}(\mathcal{I}_{h,t}^{(\ell)},{label}_{h,t}^{(\ell)}))\cap_{[\bm{I}\;\bm{0}]} \tilde{\mathcal{Z}}^{(\ell)}_{h,t}$\tcp*{Hidden layer ReLU graph}
$\tilde{\mathcal{H}}_{t}^{(\ell)}$ $\leftarrow$ $[\bm{0}\;\bm{I}] \cdot \tilde{\mathcal{G}}_{h,t}^{(\ell)}$\tcp*{Hidden state set}
}
$\tilde{\mathcal{Z}}_{y,t} \leftarrow \bm{W}_{y}\tilde{\mathcal{H}}_{t}^{(L)} +\bm{v}_{y}$;\\
$\tilde{\mathcal{G}}_{y,t}$ $\leftarrow$ $(\bm{P}\cdot\tilde{\mathcal{G}}_{ReLU}(\mathcal{I}_{y,t},{label}_{y,t}))\cap_{[\bm{I}\;\bm{0}]} \tilde{\mathcal{Z}}_{y,t}$\tcp*{Output layer ReLU graph}
$\tilde{\mathcal{H}}_{t+1}^{(0)}$ $\leftarrow$ $[\bm{0}\;\bm{I}] \cdot \tilde{\mathcal{G}}_{y,t}$\tcp*{Output set}
$\tilde{\mathcal{S}}_{x,t+1}$ $\leftarrow$ $\mathcal{C}(\tilde{\mathcal{H}}_1^{(0)},\tilde{\mathcal{H}}_{t+1}^{(0)})$\tcp*{State-pair set}
}
\Return{$\{\tilde{\mathcal{S}}_{x,t}\}^T_{t=2}$}
\end{algorithm}

Algorithm~\ref{alg:2} summarizes the procedure for computing over-approximated state–pair sets using the triangle relaxation approach. In Line~1, we first approximate the interval bound for each ReLU unit. This can be done by simply propagating an interval hull of the domain $\X$ across all layers and time steps. Alternatively, we apply CROWN to get a tighter approximation by feeding the parameters of the RNN $\mt{\pi}$, the domain set $\X$, and the horizon $T$. While tighter bounds may improve the scoring accuracy and thus reduce the conservatism of the approximated state-pair sets, they do not affect the soundness of the method. Using these precomputed interval bounds, we rank all unstable ReLU units across all layers and all time steps by the score in \eqref{equ:score} and label the top $N_b$ units with $0$; all remaining unstable units are labeled $1$. The subsequent steps follow Algorithm~\ref{alg:1}.

\begin{proposition}\label{prop2:over-state-pair}
Consider the closed-loop RNN system \eqref{equ:sys} and any given HZ $\X \subset \mathbb{R}^n$, let $\{\tilde{\mathcal{S}}_{x,t}^{N_b}\}^T_{t=2}$ be the output of Algorithm \ref{alg:2} with binary limit parameter $N_b$, then for every $t\in\{2,\dots,T\}$, the following hold true:\\
(i) The $t$-th step state-pair set over $\X$ is over-approximated by $\tilde{\mathcal{S}}_{x,t}^{N_b}$, i.e., $\tilde{\mathcal{S}}_{x,t}^{N_b}\supseteq \mathcal{S}_x(\X,t)$;\\
(ii) If $N_b$ is no less than $N_t$ (i.e., the total number of unstable ReLU activations at step $t$), then $\tilde{\mathcal{S}}_{x,t}^{N_b}= \mathcal{S}_x(\X,t)$;\\
(iii) The size of $\tilde{\mathcal{S}}_{x,t}^{N_b}$ decreases monotonically with $N_b$, i.e., $\tilde{\mathcal{S}}_{x,t}^{N_b} \supseteq \tilde{\mathcal{S}}_{x,t}^{N_b'}$ if $0 \leq N_b \leq N_b'$.
\end{proposition}

\begin{proof}
Properties (i) and (ii) follow immediately from the fact that $\tilde{\mathcal{G}}_{ReLU}(\mathcal{I}_{h,t}^{(\ell)},{label}_{h,t}^{(\ell)}) \supseteq {\mathcal{G}}_{ReLU}(\mathcal{I}_{h,t}^{(\ell)})$, $\tilde{\mathcal{G}}_{ReLU}(\mathcal{I}_{y,t},{label}_{y,t}) \supseteq {\mathcal{G}}_{ReLU}(\mathcal{I}_{y,t})$ and the set operations used in the Algorithm \ref{alg:2} are monotone with respect to set inclusion. Therefore, we have $\mathcal{S}_x(\X,t)\subseteq \tilde{\mathcal{S}}_{x,t}^{N_b}$. Moreover, if $N_b\ge N_t$, then all unstable ReLUs are represented exactly which yields $\tilde{\mathcal{S}}_{x,t}^{N_b}=\mathcal{S}_x(\X,t)$.
For property (iii), as the relaxed ReLU units are selected based on their corresponding sorted scores, the set of relaxed units strictly decreases as the binary limit parameter grows. Since the exact ReLU graph is a subset of its relaxation, replacing some relaxed graphs (under $N_b$) by exact ones (under $N_b'$) shrinks the resulting set. Therefore $\tilde{\mathcal{S}}_{x,t}^{N_b'} \supseteq \tilde{\mathcal{S}}_{x,t}^{N_b}$.
\end{proof}


Naturally, the over-approximated state–pair sets induce over-approximations of the reachable sets, as formalized in the following theorem. The proof follows directly from Proposition \ref{prop2:over-state-pair} and is therefore omitted.
\begin{theorem}\label{thm:relax}
Let $\{\tilde{\mathcal{S}}_{x,t}^{N_b}\}^T_{t=2}$ be the over-approximations of the state-pair sets over the domain $\mathcal{X}$ computed using Algorithm \ref{alg:2} with binary limit parameter $N_b$, then for every $t\in\{2,\dots,T\}$, the following statements hold true:\\
(i) Given the initial set $\mathcal{X}_1 \subseteq \mathcal{X} $, an over-approximated FRS can be computed by $\tilde{\mathcal{R}}_t(\mathcal{X}_1) = [\bm 0_{n\times n} \; \bm I_n]\cdot (\tilde{\mathcal{S}}_{x,t}^{N_b}\cap_{[\bm I_n \; \bm 0_{n\times n}]} \X_1) \supseteq {\mathcal{R}}_t(\mathcal{X}_1)$.\\
(ii) Given the target set $\mathcal{T} \subseteq \mathcal{X} $, an over-approximated BRS can be computed by $\tilde{\mathcal{P}}_t(\mathcal{T}) = [\bm I_n \; \bm 0_{n\times n}]\cdot (\tilde{\mathcal{S}}_{x,t}^{N_b}\cap_{[\bm 0_{n\times n}\; \bm I_n]} \mathcal{T}) \supseteq {\mathcal{P}}_t(\mathcal{T})$.\\
(iii) The over-approximation in (i) and (ii) becomes exact when $N_b\geq N_t$.\\
(iv) The sizes of $\tilde{\mathcal{R}}_t(\mathcal{X}_1)$ and $\tilde{\mathcal{P}}_t(\mathcal{T})$ decrease monotonically with $N_b$.
\end{theorem}


\begin{example}\label{exp1}
We consider a simple closed-loop RNN system to illustrate the size of the over-approximated reachable sets decreases monotonically with the binary limit parameter $N_b$. Specifically, we train a one-layer RNN with 10 hidden neurons to approximate a double integrator system controlled by an MPC policy. Figure~\ref{fig:simpleRNN} shows the over-approximated FRSs at step 5 (i.e., $\tilde{\mathcal{R}}_5$), for five different values of $N_b$.
\end{example}

\begin{figure}[!h]
    \centering
        \includegraphics[width=0.9\linewidth]{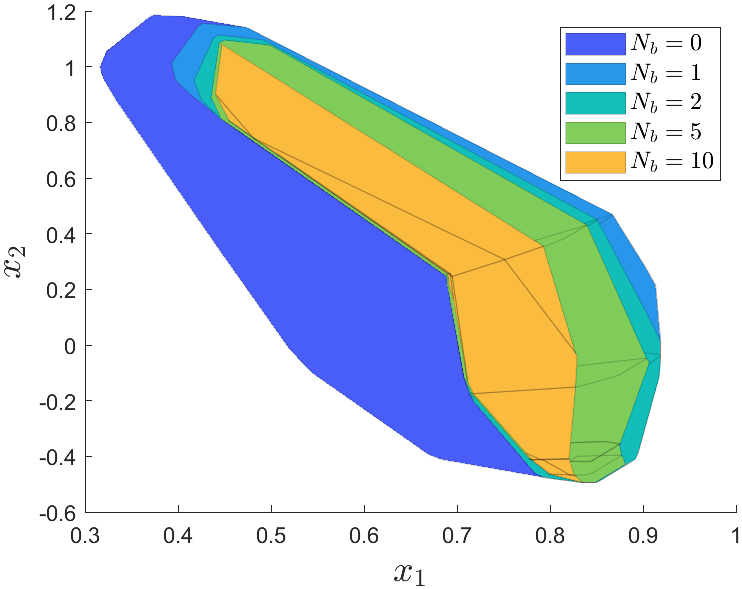}
        \caption{Over-approximated FRSs for a closed-loop RNN in Example \ref{exp1}, with varying values of binary limit parameter $N_b$. The size of the over-approximated reachable sets decreases monotonically as $N_b$ increases. }
    \label{fig:simpleRNN}
    \vskip -3mm
\end{figure}

\begin{remark}
Assume $N_b\leq N_t$, the set complexity of $\tilde{\mathcal{R}}_t(\mathcal{X}_1)$ is given by 
$\tilde n_{g,r} = n_{g,x} + 5 N_t - N_b +n_{g,1}$, $\tilde n_{b,r} = n_{b,x} + N_b+n_{b,1}$, $\tilde n_{c,r} = n_{c,x} + 3 N_t+n+n_{c,1}$,
and the set complexity of $\tilde{\mathcal{P}}_t(\mathcal{T})$ is given by 
$\tilde n_{g,p} = n_{g,x} + 5 N_t - N_b+n_{g,\tau}$, $\tilde n_{b,p} = n_{b,x} + N_b+n_{b,\tau}$, $\tilde n_{c,p} = n_{c,x} + 3 N_t+n + n_{c,\tau}$.
\end{remark}

\begin{remark}
The triangle relaxation technique has been widely used in prior work on ReLU-activated neural networks (e.g., \cite{zhang2024reachability,choi2025reachability,singh2019abstract}). Most existing methods, however, perform the ranking and selection of relaxed ReLU units on a layer by layer basis, guided by heuristic metrics. In contrast, our approach ranks ReLU units across all layers and preserves those with the highest triangle-area scores for exact representation. This formulation enables us to address the relaxation problem for the entire RNN in a single stage while directly controlling the complexity of the resulting reachable sets via a binary limit parameter. Moreover, our framework can readily incorporate other ReLU-ranking heuristics, such as BABSR \cite{bunel2020branch} and FSB \cite{de2021improved}. 
\end{remark}

\subsection{Safety Verification of Closed-loop RNN }\label{sec:application}

In this subsection,  we use
the reachable sets computed earlier to  derive a sufficient condition
for safety verification and to  identify unsafe sequences
in closed-loop RNNs.


Consider an initial set $\X_1 \subset \X$ and an unsafe set $\mathcal{O} \subset \X$ such that $\X_1 \cap \mathcal{O} = \emptyset$. The safety verification problem is to determine whether all trajectories of the closed-loop system \eqref{equ:sys}, starting from $\X_1$, remain outside $\mathcal{O}$ over a time horizon $T$. Assuming both $\X_1$ and $\mathcal{O}$ are represented as HZs, this problem can naturally be addressed using either forward or backward reachability analysis.

For the forward reachability approach, we take the initial set $\X_1$ as the input in Algorithm \ref{alg:2} and compute a sequence over-approximated FRSs $\tilde{\mathcal{R}}_2(\X_1), \tilde{\mathcal{R}}_3(\X_1), \dots, \tilde{\mathcal{R}}_T(\X_1)$. Safety verification then reduces to checking whether these sets intersect with $\mathcal{O}$. By Lemma \ref{lemma:inter} and \cite[Lemma 1]{zhang2023backward}, this amounts to checking the emptiness of the constructed HZ intersections, which requires solving $T-1$ MILPs.

Alternatively, applying Algorithm \ref{alg:2} and Theorem \ref{thm:relax} to the unsafe set $\mathcal{O}$ yields a sequence of over-approximated BRSs $\tilde{\mathcal{P}}_2(\mathcal{O}), \tilde{\mathcal{P}}_3(\mathcal{O}), \dots, \tilde{\mathcal{P}}_T(\mathcal{O})$. Safety can then be verified by checking whether these sets intersect the initial set $\X_1$.

The safety verification of closed-loop RNN systems is summarized in the following proposition, whose proof is omitted due to space limitations.

\begin{proposition}\label{prop:safety-verify}
    Suppose that an initial state set $\mathcal{X}_1 \subset \mathcal{X}$ and an unsafe set $\mathcal{O}\subset \mathcal{X}$ are both HZs. Let $\{\tilde{\mathcal{R}}_t(\X_1)\}_{t=2}^T$ and $\{\tilde{\mathcal{P}}_t(\mathcal{O})\}_{t=2}^T$ be the approximated forward and backward reachable sets computed using Algorithm \ref{alg:2} and Theorem \ref{thm:relax} with binary limit parameter $N_b$, then all the trajectories of the closed-loop RNN system \eqref{equ:sys} starting from $\mathcal{X}_1$ can avoid the unsafe region $\mathcal{O}$ within $T$ steps, if either:
    
    (i) For $t = 2,\dots,T$, $\tilde{\mathcal{R}}_t(\X_1) \cap \mathcal{O} = \emptyset$.

    (ii) For $t = 2,\dots,T$, $\tilde{\mathcal{P}}_t(\mathcal{O}) \cap \mathcal{X}_1 = \emptyset$.
\end{proposition}

If $N_b \geq N_t$, the two conditions above are both necessary and sufficient, as the computed reachable sets are exact.

For safety verification problems, forward reachability is generally preferable since $\X_1$ can be directly used as the input set in Algorithm \ref{alg:2}. As $\X_1$ is usually much smaller than the full domain $\X$, this avoids unnecessary complexity in the reachable sets. By contrast, backward reachability typically requires an a priori set, often chosen as $\X$, to construct the associated state-pair sets. Nonetheless, if safety verification fails (i.e., there exist trajectories from $\X_1$ that reach $\mathcal{O}$ within $T$ steps), one can use BRSs to explicitly construct unsafe trajectories by choosing any initial state that can lead to safety violations and propagate it forward through the RNN. Specifically, if ${\mathcal{P}}_t(\mathcal{O}) \cap \X_1 \neq \emptyset$, then the unsafe state sequences can be represented as $\mathcal{S}_{unsafe} = \{(\x_1,\x_2,\dots,\x_t) \in\mathbb{R}^{t\cdot n}\;|\; \x_1 \in {\mathcal{P}}_t(\mathcal{O}) \cap \mathcal{X}_1, \x_{k+1} = \mt{\pi}(\x_k), k\in[t-1] \}$.

\section{Simulation Results}\label{sec:example}

We use the following example to demonstrate the computation of forward and backward reachable sets using Algorithm \ref{alg:2} and Theorem \ref{thm:relax}. The implementation was carried out in Python and executed on a desktop with an Intel Core i5-11400F CPU and 32GB of RAM. The code is available at \href{https://github.com/wisc-arclab/reachability-RNN-HZ}{https://github.com/wisc-arclab/reachability-RNN-HZ}.


Consider the following mass-spring-damper system consisting of 2 carts:
$$
    \dot{\x}=\mat{\bm 0_{2 \times 2} & \bm I_{2} \\
-M^{-1} K & -M^{-1} C} \x+\mat
{\bm 0_{2 \times 2} \\
M^{-1}}\u,
$$
where the state $\x = [x_1 \;x_{2}\;\dot{x}_1\;\dot{x}_{2}]^\top\in\mathbb{R}^{4}$ contains the position and velocity of the carts, the control input $\u = [f_1\;f_{2}]^\top\in\mathbb{R}^{2}$ combines the external forces applied on each cart, and $K,M,C$ are parameter matrices the same as those given in \cite{jouret2023safety}. The continuous-time dynamics is discretized with a sampling time of $\Delta t = 0.1$ seconds. An MPC controller is designed to stabilize the carts around the equilibrium point while enforcing state and input constraints. Using randomly generated sequences of system states and inputs, we train a one-layer RNN with $4$ memory units to approximate the discretized state-space model, and another one-layer RNN with $8$ memory units to approximate the MPC controller. These two RNNs can be stacked into a two-layer RNN with $\{4,8\}$ memory units, yielding a closed-loop RNN system in the form of \eqref{equ:sys}.

We then compute FRSs up to $T=5$ using Theorem~\ref{thm:relax}, starting from the initial set $\X_1 = \li -1.5,-1\ri \times \li 0.5,1\ri \times \li 0.25,0.5\ri \times \li 0.25,0.5\ri$. For $t=2,\dots,5$, the numbers of unstable ReLU units are $N_t = 3,5,6,9$, respectively. Setting $N_d=9$ yields the exact FRSs $\mathcal{R}_2(\X_1),\dots,\mathcal{R}_5(\X_1)$, while setting $N_d=2$ produces the over-approximated FRSs. Figure~\ref{fig:forward5} shows the projections of these sets onto the $x_1$–$x_2$ plane. As expected, the exact reachable sets lie entirely within their over-approximations.

To demonstrate the use of BRSs, we specify an unsafe region $\mathcal{T} = \li -0.1,0.1\ri \times \li -0.04,0.04\ri \times \li 0.06,0.12\ri \times \li -0.36,-0.28\ri$ and compute a 5-step BRS $\mathcal{P}_5(\mathcal{T})$ which contains all unsafe initial states in $\X_1$ that can reach the unsafe region in 5 time steps. Based on this, we construct the unsafe sequence set $\mathcal{S}_{unsafe} = \mathcal{P}_5(\mathcal{T})\times \mathcal{R}_2(\mathcal{P}_5(\mathcal{T})) \times \mathcal{R}_3(\mathcal{P}_5(\mathcal{T}))\times \mathcal{R}_4(\mathcal{P}_5(\mathcal{T}))\times\mathcal{R}_5(\mathcal{P}_5(\mathcal{T}))$, as shown in Figure \ref{fig:backward}.


\begin{figure}[!h]
    \centering
        \includegraphics[width=0.95\linewidth]{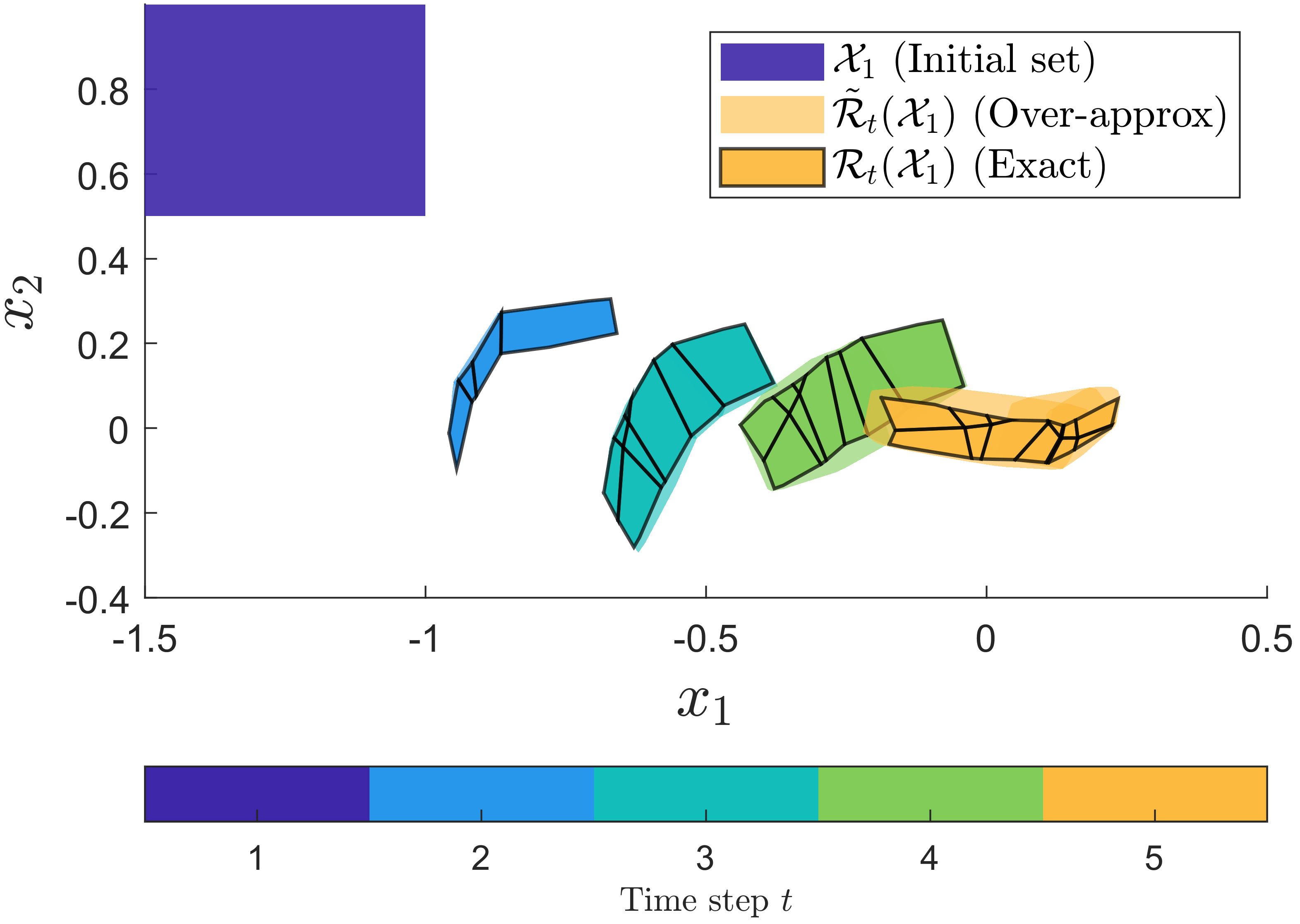}
        \caption{FRSs of the mass–spring–damper system: exact sets shown with black boundaries, and over-approximated sets shown without boundaries.}
    \label{fig:forward5}
\end{figure}

\begin{figure}[!h]
    \centering
        \includegraphics[width=0.95\linewidth]{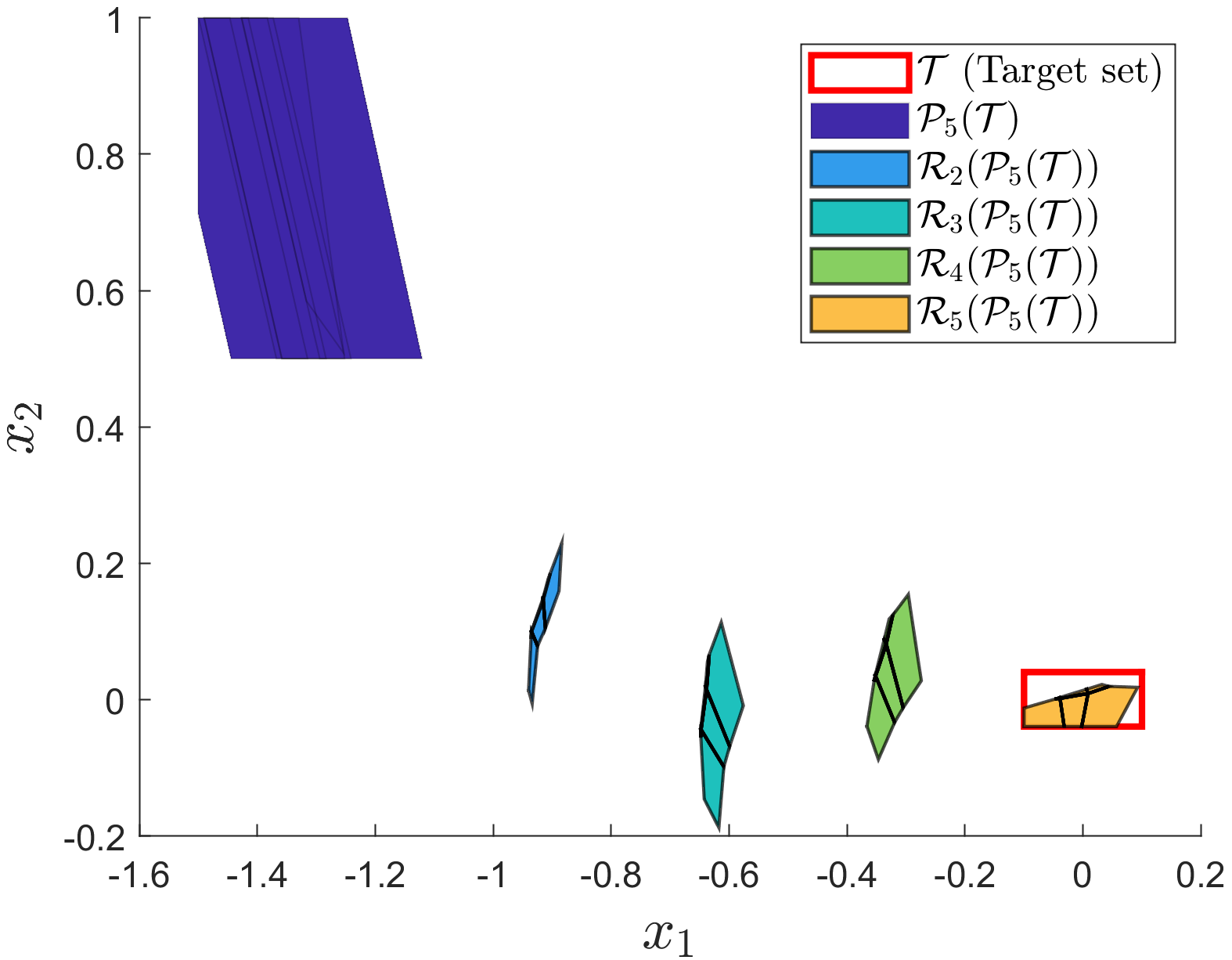}
        \caption{BRSs and the corresponding unsafe sequence set.}
    \label{fig:backward}
\end{figure}

\section{Conclusion}\label{sec:concl}
We proposed a hybrid zonotope–based framework for the forward and backward reachability analysis of closed-loop RNN systems. By constructing state-pair sets, the method derives exact HZ representations of RNN systems without unrolling. A tunable relaxation scheme was introduced based on triangle-area ranking to enable explicit trade-offs between computational complexity and approximation accuracy. We also proposed a sufficient condition for the safety verification of closed-loop RNN systems via forward and backward reachable sets. 

While this work focuses on systems in which both the dynamics and the controller are represented by RNNs, the approach can be readily extended to broader settings, including nonlinear plant models and other neural network architectures such as FNNs and CNNs studied in our prior work \cite{zhang2024reachability,zhang2025efficient}. In addition, improving the computational efficiency will be another direction for future work.


\bibliographystyle{IEEEtran}
\bibliography{ref}

\end{document}